\documentclass[12pt,a4paper,leqno]{amsart}
\usepackage{amssymb,hyperref}
\usepackage{array,float}
\usepackage{amsmath,amsthm,amsbsy,amscd,mathrsfs}
\usepackage{graphicx}
\usepackage{mathtools}

\makeatletter
\newcommand\suchthat{%
 \@ifstar
  {\mathrel{}\middle|\mathrel{}}
  {\mid}%
}
\makeatother

\newtheorem{thm}{Theorem}[section]
\newtheorem{cor}[thm]{Corollary}
\newtheorem{lem}[thm]{Lemma}
\newtheorem{prop}[thm]{Proposition}

\theoremstyle{definition}
\newtheorem{defi}[thm]{Definition}
\newtheorem{re}[thm]{Remark}

\theoremstyle{remark}
\newtheorem{alg}[thm]{Algorithm}

\numberwithin{equation}{section}

 \DeclareMathOperator{\Pic}{Pic}

 \DeclareMathOperator{\Tr}{Tr}  
 \DeclareMathOperator{\co}{covol}  
  \DeclareMathOperator{\vo}{vol} 
\DeclareMathOperator{\disc}{disc}

\title[On reduced Arakelov divisors of real quadratic fields]{ON REDUCED ARAKELOV DIVISORS OF REAL QUADRATIC FIELDS}

\author[Ha Thanh Nguyen Tran]{Ha Thanh Nguyen Tran} 
\address{Department of Mathematics and Systems Analysis,
Aalto University School of Science\\
Otakaari 1, 02150 Espoo, Finland.}
\email{hatran1104@gmail.com}


\begin{document}

\maketitle

\begin{abstract}
We generalize the concept of reduced Arakelov divisors and define $C$-reduced divisors for a given number $C \geq 1$. These $C$-reduced divisors have remarkable properties which are similar to the  properties of reduced ones. 
In this paper, we describe an algorithm to test whether an Arakelov divisor of a real quadratic field $F$ is $C$-reduced in time polynomial in $\log|\Delta_F|$ with $\Delta_F$ the discriminant of $F$. Moreover, we give  an example of a cubic field for which our algorithm does not work.
\end{abstract}

\section{Introduction}
The idea of infrastructure of real quadratic fields of Shanks in \cite{ref:7} was modified and extended by  Lenstra \cite{ref:5}, Schoof \cite{ref:8} and Buchmann and Williams \cite{ref:10} to certain number fields. Finally, it was generalized to arbitrary number fields by Buchmann \cite{ref:9}. In 2008, Schoof \cite{ref:4} gave the first description of infrastructure in terms of reduced Arakelov divisors and the Arakelov class group $\Pic^0_F$ of a general number field $F$. Reduced Arakelov divisors can be used for computing $\Pic^0_F$. They form a finite and regularly distributed set in this topological group {\cite[Propostion 7.2, Theorem 7.4 and 7.7]{ref:4}}. Computing $\Pic^0_F$ is of interest because knowing this group is equivalent to knowing the class group and the unit group of $F$ (see \cite{ref:2} and \cite{ref:4}).

Schoof proposed two algorithms which run in polynomial time in $\log{|\Delta_F|} $ with $\Delta_F$ the discriminant of $F$ {\cite[Algorithm 10.3]{ref:4}}: the testing algorithm to check whether a given  Arakelov divisor $D$ is reduced,  and the reduction algorithm to compute a reduced Arakelov divisor that is close to a given divisor $D$ in $\Pic^0_F$. However, the reduction algorithm requires finding a shortest vector of the lattice associated to the Arakelov divisor, while finding a reasonably short vector using the LLL algorithm is much faster and easier than finding a shortest vector. This leads to modifications and generalizations of the definition of reduced Arakelov divisors.

One of the generalizations, which we call $C$-reduced Arakelov divisors, comes from the reduction algorithm of Schoof  {\cite[Algorithm 10.3]{ref:4}}.
With this definition, $C$-reduced Arakelov divisors are reduced in the usual sense when $C =1$, and Arakelov divisors that are reduced in the usual sense are $C$-reduced with $C = \sqrt{n}$ (see \cite{ref:4}). $C$-reduced divisors still form a finite and regularly distributed set in $\Pic^0_F$, just like the reduced divisors.

This modification, however, has a drawback, since for general number fields it is not known how to test whether a given divisor is $C$-reduced. Currently, we have a testing algorithm to do this only for real quadratic fields, in time polynomial in $\log{\left(|\Delta_F|\right)}$. It is the  main result of this paper, presented in Section 4.

In Section 2, we discuss  $C$-reduced Arakelov divisors in an arbitrary number field. Section 3 is devoted to the properties of $C$-reduced fractional ideals of real quadratic fields. An example of real cubic fields in which the testing algorithm is no longer efficient is given in Section 5.

\section{$C$-reduced Arakelov divisors}
In this section, we introduce $C$-reduced Arakelov divisors of number fields.

Let $F$ be a number field of degree $n$ and $r_1, r_2$ the numbers of real and complex infinite primes (or infinite places) of $F$, respectively. 
 Let 
 $$F_{\mathbb{R}}: = F \otimes_\mathbb{Q}\mathbb{R} \simeq \prod_{\sigma \text{ real}}\mathbb{R} \times \prod_{\sigma \text{ complex}}\mathbb{C}.$$
  Here $\sigma$'s are the infinite primes of $F$. Then $F_{\mathbb{R}}$ is an \'{e}tale $\mathbb{R}$-algebra with the canonical Euclidean structure given by the scalar product\\
 $\hspace*{3cm}\langle u, v \rangle := \Tr(u \overline{v}) $ \text{ for } $u = (u_{\sigma})_{\sigma}, v = (v_{\sigma})_{\sigma} \in F_{\mathbb{R}}$.\\
 In particular, in terms of coordinates, we have
  $$\|u\|^2= \Tr(u \overline{u})= \sum_{\sigma \text{ real} } |u_{\sigma}|^2 + 2 \sum_{\sigma \text{ complex}  } |u_{\sigma}|^2, \text{ for any } u = (u_{\sigma})_{\sigma} \in F_{\mathbb{R}}.$$
  The \textit{norm} of an element $u = (u_{\sigma})_{\sigma} $ of $F_{\mathbb{R}}$ is defined by 
 $$N(u):= \prod_{\sigma \text{ real} }u_{\sigma} \cdot \prod_{\sigma \text{ complex} } |u_{\sigma}|^2.$$

  \begin{defi}
 An \textit{Arakelov divisor} is a formal finite sum
  $$D=\sum\limits_{\mathfrak{p}}n_{\mathfrak{p}}\mathfrak{p} +\sum_{\sigma}x_{\sigma}\sigma$$
    where $\mathfrak{p}$ runs over the nonzero prime ideals in $ O_F$ and $\sigma$ runs over the infinite primes of $F$, with $n_{\mathfrak{p}} \in \mathbb{Z}$ but  $x_{\sigma} \in \mathbb{R}$.
 \end{defi}

 To each divisor $D$ we associate the \textit{Hermitian line bundle}  $ (I, u)$  
 where $I = \prod_{\mathfrak{p}}\mathfrak{p}^{-n_{\mathfrak{p}}}$ is a fractional ideal in $F$ and 
 $u = (e^{-x_{\sigma}})_{\sigma} $ is a vector in $ \prod_{\sigma}\mathbb{R}_{>0} \subset F_{\mathbb{R}}$.

 There is a natural way to associate an ideal lattice to $D$. Indeed, $I$ is embedded into $F_{\mathbb{R}}$ by the infinite primes $\sigma$. Each element $g$ of $I$ is mapped to the vector $(\sigma(g))_{\sigma}$ in  $F_{\mathbb{R}}$. 
 Since the vector $ u g: =  (u_{\sigma} \sigma(g) )_{\sigma} \in F_{\mathbb{R}}$, we can define 
  $$\|g\|_D:= \|u g\|.$$
  In terms of coordinates, we have
  $$\|g\|_D^2= \sum_{\sigma \text{ real } }  u_{\sigma}^2 |\sigma(g)|^2 + 2\sum_{\sigma \text{ complex }} |u_{\sigma}|^2|\sigma(g)|^2.$$
 With this metric, $I$ becomes an ideal lattice in $F_{\mathbb{R}}$. We call $I$ \textit{the ideal lattice associated to $D$}. The vector $u$ has the role of a metric for $I$. Hence we make the following definition.

 \begin{defi}
 Let $I $ be a fractional ideal in $F$ and let $u$ be in  $F_{\mathbb{R}}^*$. 
 The \textit{length of an element $g$ of $I$ with respect  to the metric $u$} is defined by $\|g\|_u:= \|u g\| $.
 \end{defi}

\begin{defi}
Let $I$ be a fractional ideal. Then 1 is called  
\textit{primitive } in $I$ if $1$ belongs to $I$ and it is not divisible by any integer $\ge 2$.
\end{defi}

\begin{defi}\label{def:1}
 Let $C \geq 1$. A fractional ideal $I$ is called $C$-\textit{reduced} if:
\begin{itemize}
\item $1$ is primitive in $I$.
\item  There exists a metric $u \in \prod_{\sigma}\mathbb{R}_{>0}$  such that$  \|1\|_u \leq C\|g\|_u $  for all $ g \in I\backslash \{0\}$.
\end{itemize}
\end{defi}

\begin{re}

The second condition of Definition \ref{def:1} is equivalent to saying that there exists a metric $u$ such that with respect to this metric, the vector $1$ scaled by the scalar $C$ is a shortest vector in the lattice $I$.

\end{re}

\begin{defi} 
	Let $I$ be a fractional ideal in $F$. The\textit{ Arakelov divisor $d(I)$} is defined to be associated with the Hermitian line bundle $(I,u)$ where $u = (u_{\sigma})_{\sigma} $ with $u_{\sigma} =  N(I)^{-1/n}$ for all $\sigma$. 
\end{defi}

\begin{defi}
An Arakelov divisor $D$ is called $C$-\textit{reduced} if it has the form $D = d(I)$  for some $C$-reduced fractional ideal $I$.
\end{defi}

Now we prove the following lemma. 
\begin{lem}\label{prop:finite}
Let $I$ be a fractional ideal. If $I$ is $C$-reduced then the inverse $I^{-1}$ of $I$ is an integral ideal and its norm is at most $C^n \partial_F$ where $\partial_F = (2/\pi)^{r_2}\sqrt{|\Delta_F|}$.
\end{lem}

\begin{proof}
Since $1 \in I$, we have $I^{-1} \subset O_F$. Then $L =  N(I)^{-1/n} I$ is a lattice of covolume $\sqrt{|\Delta_F|}$ {\cite[Section 4]{ref:4}}. Consider the  symmetric, convex and bounded subset of $F_\mathbb{R}$,
$$S= \{(x_{\sigma})_{\sigma}: |x_{\sigma}| < \partial_F^{1/n} \text{ for all } \sigma \}.$$

For real $\sigma$, the segment $|x_{\sigma}| < \partial_F^{1/n}$ in $\mathbb{R}$ has length $2 \cdot \partial_F^{1/n}$. For complex $\sigma$, the disc $|x_{\sigma}| < \partial_F^{1/n}$ in $\mathbb{C}$ has area $2 \pi (\partial_F^{1/n})^2$. 
Thus, 
$$\vo(S) =(2  \partial_F^{1/n})^{r_1} \cdot ( 2 \pi  (\partial_F^{1/n})^2)^{r_2}= 2^{r_1}(2 \pi)^{r_2} \partial_F = 2^n \co(L).$$
By Minkowski's theorem, there is a nonzero element $f \in I$ such that
$$ N(I)^{-1/n} |\sigma(f)| \leq \partial_F^{1/n} \text{ for all } \sigma.$$
Since $I$ is $C$-reduced, there exists a metric $u$ such that $\|1\|_u \leq C \|f\|_u$. This implies that
$\|u\| \leq C \|u\| \max_{\sigma} | \sigma(f)| \leq C \|u\| \partial_F^{1/n}  N(I)^{1/n}$.
Hence  $N(I^{-1}) \leq C^n \partial_F$.

\end{proof}

\begin{re}\label{input}
	In this paper, given a fractional ideal $I$, we assume
	that it is represented by a matrix  with rational entries as in  {\cite[Section 4]{ref:1}} and  {\cite[Section 2]{ref:2}}.  Without loss of generality, we can also  assume that the length of the input is polynomial in $\log|\Delta_F|$.  
\end{re}

By Lemma \ref{prop:finite}, to test whether $I$ is $C$-reduced, first we can check that $ N(I)^{-1} \leq C^n \partial_F$. We have the following.

\begin{lem}\label{testnorm}
	Testing $ N(I)^{-1} \leq C^n \partial_F$ can be done in time polynomial in $\log|\Delta_F|$. 
\end{lem}
\begin{proof}

Let $M$ be the matrix representation of $I$. Since we know that 
$N(I)^{-1}=\sqrt{|\Delta_F|}/\co(I)$,
 it is sufficient to check that 
 $$|det(M)|= \co(I)> (\pi/2)^{r_2} C^n.$$
  Recall that the determinant of the matrix $M$ can be computed in time polynomial {\cite[Section 1]{ref:12}}. This reason and Remark \ref{input} imply that testing $ N(I)^{-1} \leq C^n \partial_F$ can be done in time polynomial in $\log|\Delta_F|$. 	
\end{proof}

Regarding the primitiveness of $1$ in $I$, we have the result below.

\begin{lem}\label{lem:pri} 
Let $C \geq 1$ and let $I$ be a fractional ideal containing $1$ with $ N(I)^{-1} \leq C^n \partial_F$. Then testing whether or not $1$ is primitive can be done in time polynomial in $\log{|\Delta_F|}$.
\end{lem}
\begin{proof}
Let $\{c_1, ..., c_n\}$ be an LLL-reduced $\mathbb{Z}$-basis of $O_F$ and $\{b_1, ..., b_n\}$ be an LLL-reduced $\mathbb{Z}$-basis of $I^{-1}$. Since $1 \in I$, we obtain $I^{-1} \subset O_F$ and so $b_i \in O_F$ for all $i$. Then for each $i = 1,...,n$, there exist the integers $k_{ij}$ with $j = 1,...,n$ for which $b_i = \sum_i k_{ij}c_j$. Thus, there is an integer $d$ such that $1/d \in I$ if and only if $I^{-1} \subset d O_F$. This is equivalent to $d| k_{ij}$ for all $i,j$. In other words, $d| \gcd(k_{ij}, 1 \leq i,j \leq n)$. In conclusion, $1$ is primitive in $I$ if and only if $\gcd(k_{ij}, 1 \leq i,j \leq n) = 1$. 

Since $ N(I)^{-1} \leq C^n \partial_F$, an LLL-reduced $\mathbb{Z}$-basis of $I$, the coefficients $k_{ij}$ and $\gcd(k_{ij}, 1 \leq i,j \leq n)$ can be computed in polynomial time in $\log{| \Delta_F|}$. In other words, testing the primitiveness of $1$ can be done in time polynomial in $\log{|\Delta_F|}$.
\end{proof}

By Lemma \ref{lem:pri}, we know how to test the first condition of Definition \ref{def:1}. From now on, we only consider the second condition of this definition.  

\begin{re}\label{re:u} 
Note that if $u \in \prod_{\sigma}\mathbb{R}_{>0}$ satisfies the second condition of Definition \ref{def:1}, then  $u' = \left(u_{\sigma}/N(u)^{1/n}\right)_{\sigma}\in \prod_{\sigma}\mathbb{R}_{>0}$  still satisfies that condition and  $N(u')=1$. Therefore, we can always assume that $N(u) =1$ from now on. 
\end{re}

\begin{prop}\label{pro:ub}
Let $I$ be a fractional ideal and $u$ be a vector satisfying the second condition of Definition \ref{def:1} with $N(u) =1$. Then 
$$\|u\| \leq C\sqrt{n} (2/\pi)^{r_2/n}\co(I)^{1/n}.$$
\end{prop}
\begin{proof}
Let $L = u I:= \{u f = (u_{\sigma} \cdot \sigma(f))_{\sigma}: f \in I \} \subset F_{\mathbb{R}} $. Then $L$ is a lattice with metric inherited  from $F_{\mathbb{R}}$ (see \cite{ref:4}). Since $N(u) =1$, the lattice $L$ has covolume equal to $\co(I) $ . Consider the symmetric, convex and bounded subset $S$ of $F_\mathbb{R}$
$$S= \{(x_{\sigma}): |x_{\sigma}| < (2/\pi)^{r_2/n}\co(I)^{1/n} \text{ for all } \sigma \}.$$
We have 
$$\vo(S) = 2^{r_1}(2 \pi)^{ r_2} (2/\pi)^{r_2} \co(I) = 2^n \co(L).$$
 By Minkowski's theorem, there is a nonzero element $f \in I$ such that
$$ u_{\sigma} |\sigma(f)| \leq (2/\pi)^{r_2/n}\co(I)^{1/n} \text{ for all } \sigma.$$
So $$\|uf\| \leq \sqrt{n} (2/\pi)^{r_2/n}\co(I)^{1/n}.$$
Because $u$ satisfies the second condition of Definition \ref{def:1}, we have $\|u\| \leq C\|uf\|$. The proposition is then proved.
\end{proof}

\section{$C$-reduced Arakelov divisors of real quadratic fields}\label{sec:3}
In this part, fix $C \geq 1$ and a real quadratic field $F$ with the discriminant $\Delta_F$,  we will describe what $C$-reduced ideals look like, and we will investigate their properties.

Here and in the rest of the paper, we often  identify an element $g$ of fractional ideals with its image $(\sigma(g))_{\sigma} \in F_{\mathbb{R}}$. Thus, elements of fractional ideals of real quadratic fields have the form $g=(g_1, g_2) \in F_{\mathbb{R}} \cong \mathbb{R}^2$.

\begin{re}\label{rem_identityideal}
Let $F$ be an imaginary quadratic field and let $I$ be a fractional ideal of $F$. Then an element $g \in I$ can be identified with its image $g \in F_{\mathbb{R}} \cong \mathbb{C}$. The second condition of Definition \ref{def:1} is equivalent to: there exists $u \in \mathbb{R}_{>0}$  such that $  |u| \leq C|u g|$ for all $ g \in I\backslash \{0\}$. Since $u$ is a positive real number, this is equivalent to that $1/C \leq |g|$ for all $ g \in I\backslash \{0\}$.
In other words, the shortest vectors of $I$ have length at least $1/C$. 
In addition, the first vector in an LLL reduced basis of $I$ is also its shortest vector; finding this vector can also be done in polynomial time. This together with Lemma \ref{lem:pri} shows that whether a given ideal of an imaginary quadratic field is $C$-reduced can be tested easily and in polynomial time. Therefore, in this section, we only consider $C$-reduced ideals of real quadratic fields. 
\end{re}

 \subsection{A geometrical view of reduced ideals in real quadratic cases}\label{sec:geo}
We have $F_{\mathbb{R}} \cong \mathbb{R}^2$. Let $I$ be a fractional ideal of $F$ and $S_1$ be the square centered at the origin of $F_{\mathbb{R}}$ which has a vertex $(1/C,1/C)$. We have the following result.

\begin{prop}\label{pr:geo} The second condition in Definition \ref{def:1} can be restated as follows.
There exists an ellipse $E_4$, centered at the origin and passing through the vertices of $S_1$, whose interior does not contain any nonzero points of the lattice $I$. 
\end{prop}

\begin{proof}
	It is easy to see by writing down the condition $\|u\| \leq C\|uf\|$ in terms of the coordinates of $u$ and $f$.
\end{proof}

\begin{prop}\label{pr:case} If $I$ has some nonzero element in the square $S_1$ then the ellipse $E_4$ described in Proposition \ref{pr:geo} does not exist. On the other hand, $E_4$ exists when the shortest vectors of $I$ have length at least $\sqrt{2}/C$. 
\end{prop}
\begin{proof}
\begin{figure}[h]
 \centering
        \begin{minipage}{0.49\textwidth}
          \includegraphics[width=\linewidth]{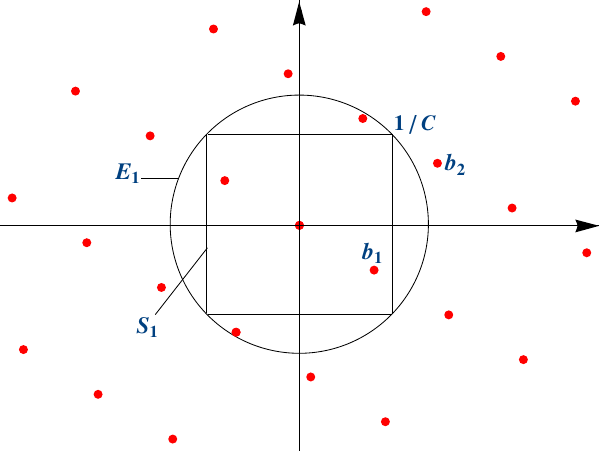}
          \caption{The shortest vectors of $I$ are inside the square $S_1$.}
          \label{pic:a}
        \end{minipage}
        \begin {minipage}{0.49\textwidth}
         \includegraphics[width=\linewidth]{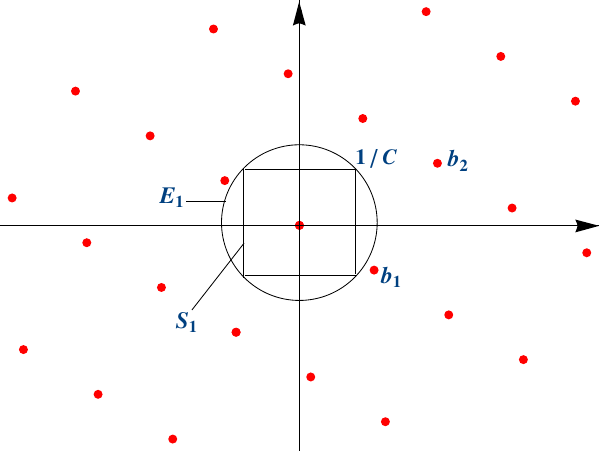}
         \caption{The shortest vectors of $I$ are outside the circle $E_1$.}
          \label{pic:b}
        \end{minipage}
        \end{figure}     
For the first case, we assume that there is a nonzero element $g$ of $I$ in the square $S_1$.  Since $S_1$ is inside $E_4$, so is $g$ (see Figure \ref{pic:a}). 
In the second case, we can take for $E_4$ the circle $E_1$ centered at the origin and of radius $\sqrt{2}/C$. Because the shortest vectors of $I$ are outside $E_1$, all the nonzero elements of $I$ are outside  $E_4$(see Figure \ref{pic:b}).
\end{proof}

\begin{re}\label{re:cases}
Proposition \ref{pr:case} does not show whether the ellipse $E_4$ exists or not in case the shortest vectors of $I$ are inside the circle $E_1$, and $I$ has no nonzero element in the square $S_1$ (see Figure \ref{pic:c}). We will discuss this case in the next sections. 
\end{re}

\begin{figure}[h]
 \centering
 \includegraphics[width=8cm]{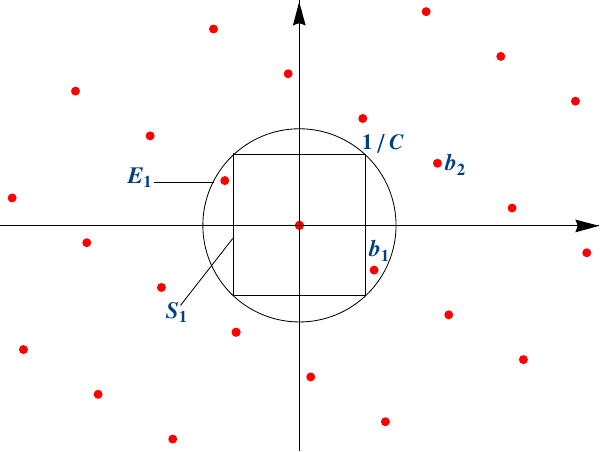}
\caption{The shortest vectors of $I$ are inside $E_1$ and $I$ has no nonzero element in $S_1$.}
 \label{pic:c}                               
\end{figure}                              
 
\subsection{Some properties of $C$-reduced ideals in real quadratic fields}\label{sec:properties}
In this section, as mentioned in Remark \ref{re:cases}, we always assume that $I$ satisfies the conditions $(\star)$ as follows.
 \[
 (\star ) 
 \begin{cases} 
 1) \hspace*{0.1cm}1 \text{ is primitive in } I. \\ 
 2) \hspace*{0.1cm}I \text{ has no nonzero element in the square } \\
  S_1 = \left\{ (x_1, x_2) \in \mathbb{R}^2: |x_1| \leq  1/C \text{ and } |x_2| \leq 1/C \text{ and } x_1^2 + x_2^2 < 2/C ^2 \right\}.\\
 3) \hspace*{0.3cm} \text{A shortest vector } f \text{ of } I \text{ has length } 1/C < \|f\| < \sqrt{2}/C.
 \end{cases}
 \]
 
 Moreover, by Remark \ref{re:u}, we can assume that the vector $u$ in Definition \ref{def:1} has the form $u = (\alpha^{-1}, \alpha) \in \left(\mathbb{R}_{>0}\right)^2 \subset  F_{\mathbb{R}} $ for some $\alpha \in \mathbb{R}_{>0}$.
 
 Let $\{b_1=(b_{1,1}, b_{1,2}), b_2=(b_{2,1}, b_{2,2})\}$ be an LLL-basis of $I$. Then $\|b_1\| =\|f\| < \frac{\sqrt{2}}{C}$.  We denote by $\{b_1^*, b_2^*\}$
 the Gram-Schmidt orthogonalization of the basis $\{b_1, b_2\}$.

Let 
$$G = \{g \in I: \left(g_1^2-\frac{1}{C^2}\right) \left(g_2^2-\frac{1}{C^2}\right) < 0 \text{ and } \|g\| < \frac{4}{ \pi} C \co(I)  \}.$$
We also set
$$G_1 = \{g \in G: g_1^2-\frac{1}{C^2}<0   \} \text{ and }  G_2 = \{g \in G: g_2^2-\frac{1}{C^2}<0   \}.$$
 So, we obtain $ G = G_1 \cup G_2$.\\
 For each $g \in G$, we define 
 $$ B(g):=\left( -\frac{C^2 g_1^2-1}{C^2 g_2^2-1}\right)^{1/4}.$$
 Then denote
 \begin{equation}
 B_{min} =
  \begin{cases} 
    \frac{1}{2 \sqrt{C}} \hspace*{3cm} \text{  if  } & G_1 = \emptyset \\ 
    \max \left\{ B(g): g \in G_1 \right\}  \text{   if } & G_1 \neq \emptyset.
    \end{cases}
 \end{equation} 
 \begin{equation}
  B_{max} =
  \begin{cases} 
    2 \sqrt{C} \hspace*{2.8cm} \text{  if  } & G_2 = \emptyset \\ 
    \min \left\{ B(g): g \in G_2 \right\}  \text{   if } & G_2 \neq \emptyset.
    \end{cases}
 \end{equation}

  Let $G' = \{g \in G: B(g) = B_{max} \text{ or } B(g) = B_{min} \}$. Then because of assumption $(\star)$, the vector $b_1$ is in $G$. Thus,  $G'$ is nonempty. 

The most important result in this paper is the following proposition.

\begin{prop}\label{lem:u}
The ideal $I$ is $C$-reduced if and only if  $ B_{min} \leq B_{max}$.
\end{prop}

We prove this proposition after proving some results below. First, we establish a property of the ellipses $E_4$ described in Section \ref{sec:geo}.

\begin{prop}\label{lem:B}
Assume that $E_4: \frac{X_{1}^2}{a_{1}^2} + \frac{X_{2}^2}{a_{2}^2} =1 $ with $a_1 >0 $ and $a_2>0$ is an ellipse satisfying the conclusion of Proposition \ref{pr:geo}. In other words, $E_4$ has its center at the origin, passes through the vertices of $S_1$ and the interior contains no nonzero points of the lattice $I$. Then:
\begin{itemize}
\item[i)] The coefficients $a_{1}$ and $a_2$ are bounded by $\frac{4}{ \pi} C \co(I)$.  
\item[ii)] $E_4$ is inside the circle $E_5$ of radius $\frac{4}{ \pi} C \co(I)$ centered at the origin.
\end{itemize}
\end{prop}

\begin{figure}[htbp]
	\centering
	\includegraphics[width=8cm]{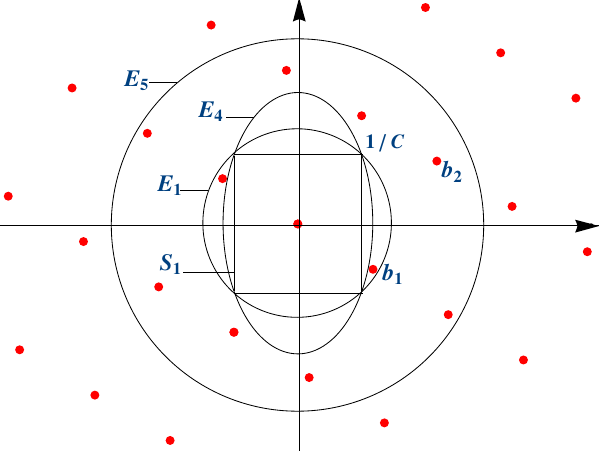}
	\caption[Picture 4]{Circle  $E_5$ and ellipse $E_4$.}
		\label{fig:circle}
\end{figure}

\begin{proof}  
 Since $E_4$ passes through the vertex $ (1/C, 1/C )$ of $S_1$, its coefficients satisfy $a_1>1/C$ and $a_2>1/C$ . We also know that $\vo(E_4) = \pi a_1 a_2$. Hence
 $$a_1 =  \frac{\vo(E_4)}{\pi a_2 } < \frac{1}{ \pi} C \vo(E_4).$$
 In addition, the ellipse $E_4$ is a symmetric, convex and bounded set whose interior contains no nonzero points of the lattice $I$, hence it must have volume less than $2^2 \co(I)$ by Minkowski's theorem. As a consequence,
 $$a_1  < \frac{4}{ \pi} C \co(I) .$$
 By symmetry, we also have this bound for $a_2$. Thus, the first statement of the proposition is obtained. The second one follows from the first.
\end{proof}

We have another equivalent condition to Definition \ref{def:1} as follows.
 
 \begin{prop}\label{lem:limit2}  The second condition of Definition \ref{def:1} is equivalent to the following: there exists a metric $u \in \left(\mathbb{R}_{>0}\right)^2$  such that 
    for all $ g \in G $, we have  $  \|1\|_u \leq C\|g\|_u $.
    \end{prop}
    
\begin{proof}
Let $g = (g_1, g_2)$ be a nonzero element of $I$. If $\|g\| \geq (4/ \pi) C \co(I) $ then $g$ is outside the circle $E_5$. By Proposition \ref{lem:B}, $g$ is also outside any ellipse $E_4$ (see Figure \ref{fig:circle}). Using this and the equivalent condition of Proposition \ref{pr:geo}, we obtain: a vector $u$ satisfies Definition \ref{def:1} if and only if $\|u\| \leq C \|ug\|$
for all $ g \in I \backslash \{0\}$ with $\|g\| < (4/ \pi) C \co(I)$. 

On the other hand, if  $|g_1| \geq 1/C$ and  $|g_2| \geq 1/C$, then $g$ satisfies $\|u\| \leq C \|ug\|$ for any $u \in \left(\mathbb{R}_{>0}\right)^2$. 
 Therefore, it is sufficient to consider the elements $g$ such that $|g_1|<1/C $ or $|g_2| < 1/C$
 to show the existence of $u$.
 
 Moreover, $I$ contains no nonzero elements of $S_1$, so $ g \notin  \{ (x_1, x_2) \in \mathbb{R}^2: |x_1| \leq  1/C \text{ and } |x_2| \leq 1/C \text{ and } x_1^2 + x_2^2 < 2/C ^2\} $. 
 
 Combining these conditions, we obtain the conclusion.
\end{proof}

The ideal $I$ with properties $(\star)$ mentioned at the beginning of this section has bounded covolume.  Explicitly, we obtain the following.
\begin{prop}\label{pr:volI}
The covolume of $I$ is bounded by $2/C$. 
\end{prop}
\begin{proof}
Since $1$ is in $I$, there exist integers $m_1$ and $m_2$ such that $1 = m_1 b_1 + m_2 b_2$. If $m_2 = 0$ then  $1= m_1 b_1$ so $1/m_1 = b_1 \in I$. Because $1$ is primitive in $I$, we must have $m_1 = \pm 1$.  Thus $ \|b_1\| = \|1\| = \sqrt{2} \geq \sqrt{2}/C$ for any $C \geq 1$. This contradicts the fact that the length of the shortest vectors of $I$ is strictly less than $\sqrt{2}/C$. Hence $m_2 \neq 0$.

We have $\|b_2^*\| \leq \frac{1}{|m_2|} \|1\| \leq \sqrt{2}$. 
 This leads to the following.
 $$ \co(I) = \|b_1\| \|b_2^*\|  < \frac{\sqrt{2}}{C} \times \sqrt{2} = \frac{2}{C}.$$  
\end{proof}

By this proposition and Proposition \ref{lem:B}, we obtain the corollary below.
\begin{cor}\label{corg}
The coefficients $a_{1}$ and $a_2$ and the radius of the circle $E_5$ in Proposition \ref{lem:B} are bounded by $8/\pi$. In addition, the set $G$ is contained in the finite set $\{g \in I: (g_1^2-1/C^2) (g_2^2-1/C^2) < 0 \text{ and } \|g\| < 8/\pi\}$.
\end{cor}

For a real quadratic field, the Proposition \ref{pro:ub} can be restated as below.

 \begin{prop}\label{lem:al}
 Assume that $u = (\alpha^{-1}, \alpha) \in \left(\mathbb{R}_{>0}\right)^2$ satisfies the second condition of Definition \ref{def:1}. Then $\|u\| \leq 2\sqrt{ C}$ and therefore
 $$\frac{1}{2 \sqrt{C}} <\alpha <  2\sqrt{ C}.$$
 \end{prop}
 \begin{proof}
 By Proposition \ref{pro:ub}, $\|u\| \leq   C \sqrt{2} \co(I)^{1/2}.$
  By Proposition \ref{pr:volI}, we have $\co(I) < \frac{2}{C}$, so $\|u\| \leq 2\sqrt{ C}$. Since $\alpha^{-1} < \|u\|$ and $ \alpha < \|u\|$, the conclusion follows.
 \end{proof}

  \begin{proof}[Proof of Proposition \ref{lem:u}]     
  Let $u = (\alpha^{-1}, \alpha) \in \left(\mathbb{R}_{>0}\right)^2$. Then from $  \|1\|_u \leq C\|g\|_u $, we get  $$\alpha^4(C^2 g_2^2-1) \geq -(C^2 g_1^2-1).$$
   Thus $\alpha \geq B(g)$ if $g \in G_1$ and $\alpha \leq B(g)$ if $g \in G_2$. 
   As $1$ is primitive in $I$, by Proposition \ref{lem:limit2}, the ideal $I$ is $C$-reduced if and only if it satisfies the following equivalent conditions: 
   
   there exists $u \in \left(\mathbb{R}_{>0}\right)^2$  such that 
   $  \|1\|_u \leq C\|g\|_u$ for all $ g \in G $, 
 \[
  \Leftrightarrow \text{ There exist }  \alpha \in \mathbb{R}_{>0} \text{  such that  }
   \begin{cases} 
   \alpha \geq B(g) &  \text{ for all } g \in G_1 \\ 
   \alpha \leq B(g)  &\text{    for all }  g \in G_2
  \end{cases}
 \]
  \[
  \Leftrightarrow \text{ There exists  } \alpha \in \mathbb{R}_{>0} \text{  such that }
  \begin{cases} 
    \alpha \geq B_{min} \\ 
    \alpha \leq B_{max}
    \end{cases}\hspace*{2.6cm}
\]
\[ \Leftrightarrow  B_{max} \geq B_{min}.\hspace*{8.5cm}
     \]
   The second equivalence comes from Proposition \ref{lem:al} and the definition of $B_{min} $ and $B_{max}$.
 \end{proof}
 
 Proposition \ref{lem:u} and \ref{lem:limit2} motivate a further investigation of properties of the sets $G$ and $G'$. We first establish a special property of the elements in $G$.

\begin{prop}\label{lem:s2} 
If $g = s_1b_1 + s_2 b_2 \in G$  then $|s_2| \leq 1 $. 
\end{prop}
\begin{proof}
Let $ g = s_1b_1 + s_2 b_2 $ in $G$. As in the proof of Proposition \ref{pr:volI}, we get $\|b_1\| < \sqrt{2}/C$ and  $\|b_2^*\| \leq \sqrt{2}$. By the properties of LLL-reduced bases,  $\|b_2\| \leq  \sqrt{2} \|b_2^*\| \leq 2 $. Therefore,
$$ \frac{4C \co(I)}{\pi} = \frac{4C \|b_1\| \|b_2^*\|}{\pi} < \frac{4\sqrt{2}\|b_2^*\|}{\pi}.$$
Now let $g^*$ be a vector of length equal to the distance from $g$ to the 1-dimensional vector space $\mathbb{R}.b_1$. In other words,  $\|g^*\| = d(g, \mathbb{R}.b_1) = |s_2| \|b_2^*\|$. 
If $|s_2| \geq 2$, then then we would have the following contradiction.
$$\|g\| \geq d(g, \mathbb{R}.b_1) = \|g^*\| \geq 2 \|b_2^*\| > \frac{4\sqrt{2} \|b_2^*\|}{\pi} > \frac{4}{ \pi} C \co(I) . $$
Thus  $|s_2| \leq 1$. 
\end{proof}

 In the next proposition, we prove that the cardinality of $G$ is bounded by a number that depends only on $C$ but not on $I$ or the number field $F$.

\begin{lem}\label{lem:carG}
The number of vectors in $G$ (up to sign)  is less than $ 17 C + 3$. 
\end{lem}  
\begin{proof}
Let $g \in G$. Then $ g = s_1 b_1 + s_2 b_2$ for some integers $s_1, s_2$. We have $\|b_1\| \geq  1/C$ and  $\|g\| < 8/\pi$ (by Corollary \ref{corg}).  This implies that 
$$ |s_1| \leq \sqrt{2} \left(\frac{3}{2}\right) \frac{\|g\|}{\|b_1\|} <   \frac{12 \sqrt{2} C}{\pi}$$
{\cite[Section 12]{ref:1}}. By Proposition \ref{lem:s2}, we obtain $ |s_2| \leq 1.$

Consequently,  the number of elements  in $G$ (up to sign) is at most  $ 3 \cdot (\frac{12 \sqrt{2} C}{\pi} + 1) $, which is less than $ 17 C + 3$. 
\end{proof}

The proposition below gives a property of elements in $G'$.

\begin{prop}\label{lem:s11}   
Let $g = s_1 b_1 +  b_2 \in G' $. Then:
\begin{itemize}
\item $|s_1| \leq 2$ or
\item $s_1 \in \{ t_1, t_2\}$ for some integers $t_1 \leq t_2$ in the interval $(-1- 2C ,  1+ 2C)$.
\end{itemize}  
 \end{prop}

\begin{proof}
It is easy to show that $b_1\in G = G_1 \cup G_2$ since $\|b_1\| \leq (4/ \pi) C \co(I)$. Here, we only prove the proposition for $b_1 \in G_1$, so $0<b_{11}<1/C$ and $1/C<|b_{12}|<\sqrt{2}/C$. For $b_1 \in G_2$ , it is sufficient to switch $b_{11}$ and $b_{12}$. In the first case, by definition of $B_{min}$, we obtain $B(b_1) \leq B_{min}$.
\begin{figure}[htbp]
\centering
	\includegraphics[width=8cm]{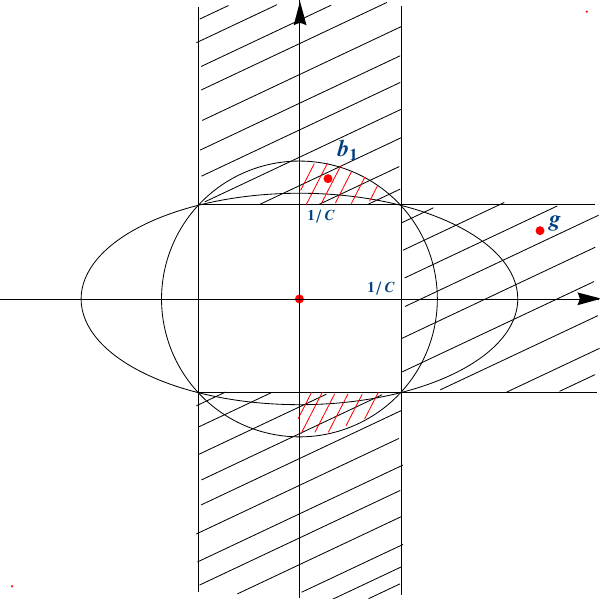}
	\caption[Picture 6]{$b_1$ is in the doubly-shaded area and $g$ is in the shaded area.}
	\label{pic:6}
\end{figure}
The element $g$  is in $G'$, and it belongs to $G_1$ or $G_2$.
\item If $g$ is in $G_1$ then $0<|g_1|<1/C$ and $|g_2|>1/C$. Since $g \in G'$ and $B(b_1) \leq B_{min}$, we also have $B(b_1) \leq B(g)$. If $\|g\| > \sqrt{2}/C$ then $B(b_1) > B(g)$, contradicting the previous inequality. So, $\|g\| \leq \sqrt{2}/C$. With this in mind and the properties of LLL-reduced bases {\cite[Section 12]{ref:1}}, we obtain
 $$|s_1| \leq \sqrt{2}\left(\frac{3}{2}\right)\frac{\|g\|}{\|b_1\|} < \sqrt{2}\left(\frac{3}{2}\right)\left(\frac{\frac{\sqrt{2}}{C}}{\frac{1}{C}}\right)= 3 \qquad \text{ so } |s_1| \leq 2.$$

If $g$ is in $G_2$ then $|g_1|>1/C$ and $|g_2|<1/C$. 
 Since $g = s_1b_1 + b_2$ and $|g_2|<1/C$, the value of $s_1$ is between  $\frac{-1/C-b_{22}}{b_{12}}$ and $\frac{1/C-b_{22}}{b_{12}}$. 
   The fact that $0<|b_{12}|<\sqrt{2}/C$ implies that the distance between these numbers 
   $$\left|\frac{-1/C-b_{22}}{b_{12}} -\frac{1/C-b_{22}}{b_{12}}\right| = \frac{2}{C|b_{12}|}$$ 
   is in the interval $(\sqrt{2},2)$. So, there exist two integers $t_1 \leq t_2$ between these  numbers. Moreover, since $1/C<|b_{12}|<\sqrt{2}/C$ and since $|b_{22}| < \|b_2\| \leq 2$ (see the proof of Proposition \ref{lem:s2}), one can easily see that 
   $$\left| \frac{\pm 1/C-b_{22}}{b_{12}} \right| < 1+ 2 C.$$
   Thus, the bounds for $s_1$ are implied, completing the proof.  

\end{proof}

\section{ Test algorithm for real quadratic fields}\label{sec:alg}
In this section, given $C\geq 1$, we explain an algorithm to test whether a given fractional ideal $I$ is $C$-reduced for a real quadratic field $F$ in time polynomial in $\log|\Delta_F|$ with $\Delta_F$ the discriminant of $F$.

By Proposition \ref{lem:u}, if we know  $B_{min} $ and $B_{max}$, then we can show the existence of a metric $u = (\alpha^{-1}, \alpha)$ in Definition \ref{def:1}. In this algorithm, we first find  all the possible elements of $G' = \{g \in G: B(g) = B_{max} \text{ or } B(g) = B_{min} \}$ and  then compute $B_{min} $ and $B_{max}$. Let $\{b_1, b_2\}$ be an LLL-basis of $I$ and $g = s_1 b_1 + s_2 b_2\in G'$. Then Proposition  \ref{lem:s2} says that $s_2 =0 $ or $s_2 = \pm 1$. By symmetry, it is sufficient to consider only the case $s_2 \in \{ 0, 1\}$.
\begin{itemize}
\item If $s_2 = 0$ then $g=b_1$.
\item If $s_2 = 1 $ then $g = s_1 b_1 + b_2$. By  Proposition \ref{lem:s11}, there are five possible values for $s_1$ in the interval $[-2,2]$ and  two  possible values $ t_1, t_2$ (with $t_1 \leq t_2$) of $s_1$ either between $\frac{-1/C-b_{22}}{b_{12}}$ and $\frac{1/C-b_{22}}{b_{12}}$ or between $\frac{-1/C-b_{21}}{b_{11}}$ and $\frac{1/C-b_{21}}{b_{11}}$. This proposition also shows that the coefficients $s_1$ have absolute values less than $1+ 2 C$.
\end{itemize}
Furthermore, by Proposition \ref{lem:al}, we have $\frac{1}{2 \sqrt{C}} <\alpha < 2 \sqrt{C}$ and so  $\frac{1}{16 C^2} <B(g)^4 < 16 C^2$ for all $g \in G$. In other words:
\[
 (\star \star ) 
  \begin{cases} 
    \text{If } |g_2| < 1/C \text{ then }\\
    \hspace*{1cm} |g_1|^2 + 16 C^2 |g_2|^2 < 16 + \frac{1}{C^2}  \text{ and } |g_2|^2 + 16 C^2 |g_1|^2 > 16 + \frac{1}{C^2}. \\ 
    \text{ If } |g_2| > 1/C \text{ then } \\
    \hspace*{1cm}|g_2|^2 + 16 C^2 |g_2|^2 > 16 + \frac{1}{C^2}  \text{ and } |g_2|^2 + 16 C^2 |g_1|^2 > 16 + \frac{1}{C^2}.
    \end{cases}
    \]

The statements in $(\star \star )$ can be applied to eliminate some elements $g$ which are not in $G'$ without having to compute $B(g)$.

Let $C \geq 1$ and let $I$ be a fractional ideal of $F$. Assume that an LLL-reduced basis $\{ b_1, b_2\}$ of $I$ is also given and change the sign if necessary to have the first component of $b_1 = (b_{11}, b_{12}) \in F_{\mathbb{R}}$ positive. In Remark \ref{input}, we assume that  the coordinates of $b_1$ and $b_2$  have at most $O((\log|\Delta_F|)^a)$ digits for some integer $a>0$. 

We have the following algorithm to test whether $I$ is $C$-reduced in time polynomial in $\log|\Delta_F|$.\\ 
  
\fbox{\begin{minipage}{0.95\textwidth}
 \begin{alg}\label{alg}  \qquad\\
   \begin{itemize}
 \item[1. ] Check if $1 \in I$ and $N(I)^{-1} < C^2 \sqrt{|\Delta_F|}$ or not.
 \item[2. ] Test whether or not $1 \in I$ is primitive.
\item[3. ] Check whether there is no nonzero element of $I$ in \\
$S_1 = \{ (x_1, x_2) \in \mathbb{R}^2: |x_1| \leq  1/C \text{ and } |x_2| \leq 1/C \text{ and } x_1^2 + x_2^2 < 2/C ^2 \}$.
\item[4. ] If $\|b_1\| \geq \sqrt{2}/C$ then $I$ is $C$-reduced. \\
If not, then find all possible elements of $G'$.
\begin{itemize}
   \item  If $0<b_{11}<1/C$ and $1/C<|b_{12}|<\sqrt{2}/C$ then compute the integers $t_1 \leq t_2$ which are  between $\frac{-1/C-b_{22}}{b_{12}}$ and  $\frac{1/C-b_{22}}{b_{12}}$ . 
   \item If $1/C<b_{11}<\sqrt{2}/C$ and $0<|b_{12}|<1/C$ then compute the integers $t_1 \leq t_2$ which are  between $\frac{-1/C-b_{21}}{b_{11}}$ and $\frac{1/C-b_{21}}{b_{11}}$.
 \end{itemize}  
   Let $ G_3 = \{b_1,  t_1 b_1 +b_2,  t_2 b_1 +b_2,  s_1b_1 + b_2 \text{ with } |s_1| \leq 2\} $.   
\item[5. ]  Remove from $G_3 $ all elements which do not satisfy $(\star \star )$.
\item[6. ] Compute $B(g)$ for all $g \in G_3$, and then $B_{max} $ and $B_{min}$.
 
   \end{itemize}
    If $B_{min} \leq B_{max}$ then $I$ is $C$-reduced. If not, then $I$ is not $C$-reduced.     
 \end{alg}

\end{minipage}}\\

 Step 3 of Algorithm \ref{alg} is done in a similar way to testing the minimality of 1 was done (cf.{\cite[Algorithm 10.3]{ref:4}}) but here 1 is replaced by $\frac{1}{C}$. In fact, we have the lemma below.
 
 \begin{lem}\label{sq}
 Step 3 of Algorithm \ref{alg} can be done by checking at most six short vectors of the lattice $I$.
 \end{lem}
 \begin{proof}
 	If $b_1$ is in $S_1$ then $I$ is not $C$-reduced. Otherwise,  $\|b_1\| > \frac{1}{C}$. Assume that $ g= s_1 b_1 +s_2 b_2$ is in $S_1$. Then $g$ has length $\|g\|  < \frac{\sqrt{2}}{C}$. 
 	
 	Since $\{b_1, b_2\}$ is an LLL-reduced basis of $I$, the coefficients $s_1$ and $s_2$ are bounded as
 	$$|s_1| \leq \sqrt{2}\left(\frac{3}{2}\right)\frac{\|g\|}{\|b_1\|} < \sqrt{2}\left(\frac{3}{2}\right)\left(\frac{\frac{\sqrt{2}}{C}}{\frac{1}{C}}\right)=3 $$
 	and $$|s_2| \leq \sqrt{2} \frac{\|g\|}{\|b_1\|} < \sqrt{2} \left(\frac{\frac{\sqrt{2}}{C}}{\frac{1}{C}}\right)=2 $$  
 	{\cite[Section 12]{ref:1}}. Therefore, the elements of $I$ which are in $S_1$ have the form $g= s_1 b_1 +s_2 b_2 $ with $|s_1| \leq 2$ and $|s_2| \leq 1$. By symmetry, it is sufficient to test at most six short elements of $I$.
 \end{proof} 

\begin{prop}
 Algorithm \ref{alg} runs in time polynomial in $\log|\Delta_F|$.
 \end{prop}
\begin{proof}
 The first step can be done in polynomial time in $\log|\Delta_F|$  by Lemma \ref{testnorm}. An LLL-reduced basis of $I$ can be computed in time polynomial in $\log |\Delta_F|$ and Step 2 can be done in time polynomial in $\log|\Delta_F|$ (see Lemma \ref{lem:pri} in Section 2). In Step 3, by Lemma \ref{sq}, it is sufficient to check few short vectors of $I$ which have length bounded by $\frac{\sqrt{2}}{C}$. Step 4 can be done by finding $2$ integer numbers $t_1, t_2$ which are in the interval $[-1 -2 C , 1+ 2 C]$.
 In Step 6, the bounds $B(g)$ are between $\frac{1}{2 \sqrt{C}} $ and  $2 \sqrt{C}$. Overall, this algorithm runs in time polynomial in $\log|\Delta_F|$.
\end{proof}

\section{ A counterexample}
By Lemma \ref{testnorm}, \ref{lem:pri} and \ref{sq}, the first three steps of Algorithm \ref{alg} can be done in time polynomial in $\log|\Delta_F|$. Essentially, the last three steps require finding all elements of $I$ in a certain subset $G$ (see Proposition \ref{lem:limit2} and Lemma \ref{lem:limit3}). Therefore, the complexity of this algorithm is proportional to the cardinality of $G$.

 For real quadratic fields,  the task can be reduced to finding all elements of the subset $G'$ of $G$ by Proposition \ref{lem:u}. Since Proposition \ref{lem:s2} and \ref{lem:s11} say that $G'$ have few elements and it is easy to compute them, Algorithm \ref{alg} works well, i.e., it runs in time polynomial in $\log|\Delta_F|$.

 However, for a number field of degree at least $3$, the set $G$ may have many elements, and we currently do not know how to reduce $G$ to a smaller subset. Therefore, an algorithm similar to Algorithm \ref{alg} would be inefficient. In other words, in bad cases, the complexity of Step 4--6 of Algorithm \ref{alg} may reach $|\Delta_F|^a$ for some $a>0$. In this section, we provide an example of a real cubic field $F$ with large discriminant $\Delta_F$ for which $G$ has at least $|\Delta_F|^{1/4}$ elements.

 Since $F$ is a real cubic field, we have $F_{\mathbb{R}} \cong \mathbb{R}^3$. Let $I$ be a fractional ideal of $F$. Then we identify each element $g \in I$ with its image $(\sigma(g))_{\sigma} =(g_1, g_2, g_3) \in F_{\mathbb{R}} \cong \mathbb{R}^3$.

 We set 
 $$\delta(I,C) = \frac{6 }{ \pi} C^2 \co(I)$$
  and let  
   $$S_1 = \left\{(x_1,x_2,x_3) \in \mathbb{R}^3: |x_i| \leq 1/C, 1 \leq i \leq 3 \text{ and } x_1^2+ x_2^2 + x_2^2 < 3/C^2 \right\},$$   
   $$G =\{g =(g_1, g_2, g_3) \in I: \|g\| < \delta(I,C) \text{ and } \hspace*{5cm}$$
   $$ \hspace*{6cm} \text{ there exists } i \text{ such that }  |g_i|<1/C  \}.$$   
   Let $E_1$ be the sphere centered at the origin of radius $\frac{\sqrt{3}}{C}$.
  As condition $(\star) $ for the quadratic case (see Remark \ref{re:cases}), we assume that $ 1$ is primitive in $I$ and $I$ contains no nonzero element of $S_1$ but the shortest vectors of $I$ are inside $E_1$.

 Proposition \ref{pr:geo} and \ref{lem:B} for the quadratic case can be naturally generalized to a real cubic field. Similar to Proposition \ref{lem:limit2}, we have the following result.

 \begin{lem}\label{lem:limit3}  The second condition of Definition \ref{def:1} is equivalent to: 
 there exists a metric $u \in \left(\mathbb{R}_{>0}\right)^3$  such that 
      $  \|1\|_u \leq C\|g\|_u $ for all $ g \in G $.
  \end{lem}

Let $\{b_1,  b_2, b_3 \} $  be an LLL-basis of $I$. We give an example with $C=1$.
\subsection{ An example}\label{sec:ex} 
Let $P(X) = 10000000019 X^3 + 10218400019 X^2 - 8813199073 X - 4923977196$ be an irreducible polynomial with a root $\beta$ and $F= \mathbb{Q}(\beta)$. Then $F$ is a real cubic field with discriminant 
 $$\Delta_F = 70862499223222398531211367826392679055149> 7 \cdot 10^{40}.$$ 
Denote by $O_F$ the ring of integers of $F$. Let $I = O_F + O_F \beta$. Then the fractional ideal $I$ has the properties that:
 	 \begin{itemize} 
 	     \item $1$ is primitive in $I$.	 	 	  	 
 	  	 \item $I$ has no nonzero element in the cube $S_1$.
 	  	 \item $b_1$ is inside $E_1$ so are the shortest vectors of $I$.
 	  	 \item The covolume of $I$ is greater than 	 
 	  	 $ 1.6 \cdot |\Delta_F|^{1/4}$. 	   	  
 	 \end{itemize} 	 
The cardinality of $G$ is at least  $ 1.7 \cdot 10^{10} >  |\Delta_F|^{1/4} $.

\subsection{How to find the above example}
We construct a real cubic field $F$ with a fractional ideal $I$ satisfying the conditions of Section \ref{sec:ex}.

   Let $C \geq 1$. Assume that $F= \mathbb{Q}(\beta)$ for some  $\beta$ of length $\|\beta\| < \sqrt{3}/C$ and outside the cube $S_1$. Let $O_F$ be the ring of integers of $F$.   
  Suppose that $I = O_F + O_F \beta$.  Then the shortest vectors of $I$ have length at most $\|\beta\| < \sqrt{3}/C$. 
   
   Denote by $P(X)= a X^3 + b X^2 + c X + d \in \mathbb{Z}[X]$ with $gcd(a,b,c,d) =1$ and $a>0$ an  irreducible polynomial that has a root $\beta$. Let 
   $$R = \mathbb{Z} \oplus \mathbb{Z}(a \beta) \oplus \mathbb{Z} (a \beta^2 + b \beta).$$
  Then $R$ is a multiplier ring. Hence it is an order of $F$ {\cite[Section 12.6]{ref:28}}.

Denote by $ \beta_1 = \beta , \beta_2 $ and $ \beta_3$ the roots of $P(X)$. 
 We can easily choose $P(X)$ such that $O_F=R$. This can be done by using the lemma below.
        
   \begin{lem}
        	If the  discriminant of $P(X)$ is squarefree then  $O_F = R$.
   \end{lem}
    \begin{proof}
        The discriminant of $P(X)$ is $\disc(P)= a^4\prod_{i<j}(\beta_i-\beta_j)^2 $ 
        {\cite[Proposition 3.3.5]{ref:11}}. By computing the discriminant of $R$, we can easily see that it is equal to $\disc(P)$. The result follows since  $[O_F: R]^2 |\disc(P)$. 	
   \end{proof}

       
   \begin{lem}
       If $O_F = R$ then $ N(I^{-1}) =a$.
   \end{lem}
       
   \begin{proof}
       Since $O_F = R = \mathbb{Z} \oplus \mathbb{Z}(a \beta) \oplus \mathbb{Z} (a \beta^2 + b \beta) $ and $I = O_F + O_F \beta$, it is easy to see that 
       $I = \mathbb{Z} \oplus \mathbb{Z} \beta \oplus \mathbb{Z} (a \beta^2)$.        	
       It leads to   $N(I^{-1})=  [I: O_F] = a $ and  the lemma is proved. 
   \end{proof}

 The next lemma says that $a$ can be chosen such that $1$ is primitive in $I$.
   \begin{lem}
       If $a$ is a prime number then $1$ is primitive in $I$.
    \end{lem}
         
     \begin{proof}
         If there is an integer $d \ge 2$ such that $1/d \in I$, then $ d^3= N(d)| N(I^{-1})=a$, impossible since $a$  is a prime. Thus, $1$ is primitive in $I$. 
     \end{proof}

Let $\{b_1 = (b_{11}, b_{12}, b_{13}), b_2 = (b_{21}, b_{22}, b_{23}), b_3 = (b_{31}, b_{32}, b_{33})\} \subset \mathbb{R}^3 \subset F_{\mathbb{R}}$ and $\{b_1^*, b_2^*, b_3^* \}$   the Gram-Schmidt orthogonalization of this basis. We have the following result, crucial to obtaining the example of Section \ref{sec:ex}.
           
\begin{prop}\label{pro:cardG}
    Let $C \geq 1$. Assume that:
    \begin{itemize}
        \item $1$ is primitive in $I$.
        \item  $I$ has no nonzero elements in the cube $S_1$.
        \item $\|b_1\| < \sqrt{3}/C$.
        \item $\co(I) \ge 10$.
    \end{itemize}
    Then the cardinality of $G$ is at least $\frac{2}{3} C^2 \co(I)$.
\end{prop}
                      
 \begin{proof}           	
As $I$ has no nonzero element in $S_1$, there is some coordinate $b_{1j}$ with $1 \leq j \leq 3$ of $b_1$ such that $|b_{1j}|\geq 1/C$.           	
Let $g =s_1 b_1 + s_2 b_2 = (g_1, g_2, g_3) $.        
We show that if $|s_2| \leq \frac{1}{3} C^2 \co(I)$ and if $s_1$ is between two numbers $\frac{1}{b_{1j}}(1/C-s_2 b_{2j} )$ and $\frac{1}{b_{1j}}(-1/C-s_2 b_{2j} )$, then $g$ is in $G$.       
                      	
 We know that  $\|b_1\| < \sqrt{3}/C$, hence $|b_{1j}| <\sqrt{3}/C$. This means that for each $s_2$, the distance between $\frac{1}{b_{1j}}(1/C-s_2 b_{2j} )$ and $\frac{1}{b_{1j}}(-1/C-s_2 b_{2j} )$ is  greater than $ 2/\sqrt{3}>1 $. Therefore there is at least one integer $s_1$  between them.　    	
           	
 The bound for $s_1$ implies that $|g_j|< \frac{1}{C}$. To prove that $g \in G$, it is sufficient to prove that $\|g\|< \delta(I,C)$.
           	
We first show that  $\|b_2\| \le \sqrt{3}$. 
Since  $1$ is in $I$, there exist integers $m_1, m_2$ and $m_3$ such that $1 = m_1 b_1 + m_2 b_2 + m_3 b_3$. If $m_3 = m_2 = 0$  then  $1= m_1 b_1$. It follows that $1/m_1 = b_1 \in I$. Since $1$ is primitive, we must have $m_1 = \pm 1$.  So, $ \|b_1\| = \|1\| = \sqrt{3} \geq \frac{\sqrt{3}}{C}$ for any $C \geq 1$. This contradicts $ \|b_1\| <\sqrt{3}/C$. As a result,  $m_3 \neq 0$ or $m_2 \neq 0$.         
If $m_3 \neq 0$, then $\|b_3^*\| \leq \frac{1}{m_3} \|1\| \leq \sqrt{3}$. By the properties of LLL-reduced bases {\cite[Section 12]{ref:1}}, we have $\|b_2^*\| \leq \sqrt{2} \|b_3^*\| \leq \sqrt{6}$. 
Then 
$$\co(I) = \|b_1\| \|b_2^*\|  \|b_3^*\| < \frac{\sqrt{3}}{C} \cdot\sqrt{6} \cdot \sqrt{3} =  \frac{3\sqrt{6}}{C},$$
contrary to the assumption that $\co(I) \ge 10$. Hence, $m_3=0$ and $m_2 \neq 0$. Consequently,     
           	$\|b_2^*\| \leq \frac{1}{|m_2|} \|1\| \leq \sqrt{3}$. 
           	
Next, we prove that $\|b_2\| \leq \frac{\sqrt{15}}{2}$. Indeed, denoting $\mu = \langle b_2, b_1 \rangle/\langle b_1, b_1 \rangle$, by the properties of LLL-reduced bases we have $|\mu| \leq \frac{1}{2}$ and $b_2 = b_2^* + \mu b_1$  {\cite[Section 12]{ref:1}}. It follows that 
$$\|b_2\|^2 = \|b_2^*\|^2 + \mu^2 \|b_1\|^2 < 3 + \frac{1}{4} \frac{3}{C^2} \leq \frac{15}{4}.$$ 
           	
Now, since $|b_{1j}|\geq \frac{1}{C}$ and $|b_{2j}| \leq \|b_2\| \leq \frac{\sqrt{15}}{2} $, the two numbers $\frac{1}{b_{1j}}(1/C-s_2 b_{2j} )$ and $\frac{1}{b_{1j}}(-1/C-s_2 b_{2j} )$ are in the interval 
$$\left[-(1+ \frac{\sqrt{15}}{2}|s_2| ) C, \hspace*{0.2cm} (1+ \frac{\sqrt{15}}{2}|s_2| ) C \right]$$
 and so is $s_1$. Therefore
$$\|g\|^2 = \|(s_1+ \mu s_2) b_1 + s_2 b_2^*\|^2 = (s_1+ \mu s_2)^2 \|b_1\|^2 + |s_2|^2 \|b_2\|^2$$ 
   $$ < \left( \left(1+ \frac{\sqrt{15}}{2} |s_2| \right) C+ \frac{1}{2} |s_2| \right)^2  \frac{3}{C^2} + 3 s_2^2  $$
   $$ \leq 3\left( 1+ \frac{1+\sqrt{15}}{2} |s_2| \right)^2   + 3 s_2^2 < [\delta(I,C)]^2 $$
  since $ |s_2| \leq \frac{1}{3} C^2 \co(I) \text{ and } \co(I) \ge 10$.
           	        	
 We have shown that $g = s_1 b_1 + s_2 b_2\in G$ for all $(s_1, s_2) \in \mathbb{Z}^2\backslash \{(0,0)\}$ with $|s_2| \leq \frac{1}{3} C^2 \co(I)$ and $s_1$ between $\frac{1}{b_{1j}}(1/C-s_2 b_{2j} )$ and $\frac{1}{b_{1j}}(-1/C-s_2 b_{2j} )$. Furthermore, if $g \in G$, then $-g \in G$. Thus,  $G$ has at least $[2 \cdot \frac{1}{3} C^2 \co(I) = \frac{2}{3} C^2 \co(I)]$ elements. 
\end{proof}

\begin{cor}
    With the assumptions in Proposition \ref{pro:cardG}, the set $G$ contains more than  $\gamma  C^2 |\Delta_F|^{1/4}$ elements for some constant $\gamma$ depending on the roots $\beta_1, \beta_2, \beta_3 $ of $P$.
\end{cor}
    
 \begin{proof}         
    	By choosing $P$ such that $O_F = R$, we have 
    	$$|\Delta_F| = \disc(R)=\disc(P)= a^4\prod_{i<j}(\beta_i-\beta_j)^2.$$
    	 Hence
    	 $$a = \frac{1}{\gamma} |\Delta_F|^{1/4} \text{ with }  \gamma = \left(\prod_{i<j}(\beta_i-\beta_j)^2 \right)^{1/4}.$$ 
    	 Consequently, 
    	$$ \co(I) = \frac{\sqrt{|\Delta_F|}}{N(I^{-1})} = \frac{|\Delta_F|^{1/2}}{a}  =\gamma |\Delta_F|^{1/4}$$ 
    	and the result follows from Proposition \ref{pro:cardG}.
 \end{proof}

   \begin{re}
     \emph{ Almost all the lattices $I$ constructed this way have no nonzero element in the cube $S_1$} as we may expect. Indeed, any element $g= s_1 b_1 + s_2 b_2 +s_3 b_3 \in     I \cap S_1$ has length at most $\sqrt{3}/C$. So, we can bound for the coefficients $s_1 s_2, s_3$ as follows  {\cite[Section 12]{ref:1}}.     
     $$|s_1| \leq  2\left(\frac{3}{2}\right)^{2}\frac{\|g\|}{\|b_1\|},  \qquad  |s_2| \leq 2\left(\frac{3}{2}\right)\frac{\|g\|}{\|b_2^*\|},      \qquad   |s_3| \leq 2\frac{\|g\|}{\|b_3^*\|}.$$ 
   Therefore,the the cardinality of $I \cap S_1$ is bounded by
      $$\frac{1}{\co(I)} \cdot \left(\frac{\sqrt{3}}{C}\right)^{3} \cdot \text{ (a constant ) }$$
      {\cite[Section 12]{ref:1}}.
   Since the covolume of $I$ is very large,  this number is very small. So, usually we can get $I$ without any nonzero elements in $S_1$.
     \end{re}

    From the idea above, some examples like the one in \ref{sec:ex} can be produced as follows.
         \begin{itemize}
        	 \item First choose the discriminant $|\Delta_F|$ of $F$ such that $|\Delta_F| > 10^{4}$ (to make sure that $\co(I) \ge 10$). 
        	 \item Choose a prime number $a \approx  |\Delta_F|^{1/4}$ (such that $1$ is primitive in $I$). 
         	 \item Chose a real vector $(\beta_1, \beta_2, \beta_3)$  outside $S_1$ and 
         	 such that 
         	 $$\frac{1}{C^2}<\beta_1^2+ \beta_2^2+ \beta_3^2<\frac{3}{C^2}.$$ 
         	 \item Find the polynomial $P(X) = a X^3 + b X^2 + c X +  d \in \mathbb{Z}[X]$ of the form $a (X-\beta_1) (X-\beta_2) (X-\beta_3)$ (this can be done by using the function \texttt{round} in \texttt{pari-gp}). Then check whether $P(X)$ is irreducible.\\
         	  Check if $\disc(P)$ is squarefree. If not then change $\beta_i$ until it is. Now $O_F = R$.
         	 \item Let $I = O_F + O_F \beta$. Compute an LLL-reduced basis $\{ b_1, b_2, b_3\} $ of $I$ and check if $\|b_1\| < \sqrt{3}/C$.
         	 \item Test whether $I$ does not have any nonzero element in $S_1$.
        \end{itemize}


\subsection*{Acknowledgements}
I would like to thank Ren\'{e} Schoof for proposing a new definition of $C$-reduced divisors as well as very valuable comments and Hendrik W. Lenstra for helping me to find the counterexample in Section 5. I am also immensely grateful to Wen-Ching Li and National Center for Theoretical Sciences (NCTS) for hospitality during a part of the time when this paper is written. I also would like to thank  Duong Hoang Dung and Chloe Martindale for useful comments. 
Moreover, I wish to thank the reviewers for their comments that helped improve the manuscript.

This research was supported by the Universit\`{a} di Roma ``Tor Vergata" and partially supported by the Academy of Finland (grants  $\#$276031, $\#$282938, and $\#$283262). The support from the European Science Foundation under the COST Action IC1104 is also gratefully acknowledged.  





\normalsize
\baselineskip=17pt




\end{document}